\documentclass[12pt, reqno]{amsart}

\usepackage{amsmath, amsthm, amscd, amsfonts, amssymb, graphicx, color}
\usepackage[bookmarksnumbered, colorlinks, plainpages]{hyperref}

\newtheorem{theorem}{Theorem}[section]
\newtheorem{lemma}[theorem]{Lemma}

\newtheorem{corollary}[theorem]{Corollary}
\theoremstyle{definition}
\newtheorem{definition}[theorem]{Definition}

\newtheorem{remark}[theorem]{Remark}

\begin{document}
\setcounter{page}{1}

\title[Some inequalities for the $q$-Extension of the  Gamma Function ]{Some inequalities for the $q$-Extension of the  Gamma Function }

\author[K. Nantomah, E. Prempeh and S. B. Twum]{Kwara Nantomah$^{*}$$^1$,  Edward Prempeh$^2$ and  Stephen Boakye Twum$^3$}

\address{$^{1}$ Department of Mathematics, University for Development Studies, Navrongo Campus, P. O. Box 24, Navrongo, UE/R, Ghana. }
\email{\textcolor[rgb]{0.00,0.00,0.84}{mykwarasoft@yahoo.com, knantomah@uds.edu.gh}}

\address{$^{2}$ Department of Mathematics, Kwame Nkrumah University of Science and Technology, Kumasi, Ghana. }

\address{$^{3}$ Department of Mathematics, University for Development Studies, Navrongo Campus, P. O. Box 24, Navrongo, UE/R, Ghana. }




\date{October 2015
\newline \indent $^{*}$ Corresponding author}

\begin{abstract}
In this paper, the authors establish some inequalities involving the $q$-extension of the classical Gamma function. These inequalities provide bounds for certain ratios  of the $q$-extended Gamma function. The procedure makes use of  geometric convexity and monotonicy properties of certain functions associated  with the $q$-extended Gamma function.
\end{abstract} \maketitle

\section{Introduction and Preliminaries}
\noindent
In recent years, the theory of inequalities has developed from a collection of isolated formulas into a vibrant independent area of research.  This is manifested by the emergence of several new journals devoted to this area of research. Particularly, inequalities involving special functions have been studied intensively by researchers across the globe. In this study, we establish some new inequalities  involving the $q$-extension of the Gamma function.  Before we present our results, let us recall  the following definitions pertaining to the results.\\

\noindent
The classical Euler's Gamma function, $\Gamma(x)$ is usually defined for $x>0$ by
\begin{equation*}\label{eqn:gamma}
\Gamma(x)=\int_0^\infty t^{x-1}e^{-t}\,dt = \lim_{n \rightarrow \infty}\left[ \frac{n!n^x}{x(x+1)\dots(x+n)} \right]. 
\end{equation*}

\noindent
It is well-known in literature that the Gamma function satisfies the following basic properties.
\begin{align}
\Gamma(x+1)&=x\Gamma(x) ,  \quad x\in R^{+} \label{eqn:functional-eqn-Gamma-reals}\\
\Gamma(n+1)&=n!, \quad n\in Z^{+}.  \label{eqn:functional-eqn-Gamma-integers} 
\end{align}

\noindent
Let $\psi(x)$ be the digamma or psi function  defined for $x>0$ as the logarithmic derivative of the Gamma function. That is,
\begin{equation*}\label{eqn:Psi}
\psi(x)=\frac{d}{dx}\ln \Gamma(x)=\frac{\Gamma'(x)}{\Gamma(x)}.
\end{equation*}

\noindent
The following series  representations hold true for $\psi(x)$, $x>0$ \cite{Abramowitz-Stegun-1965}.
\begin{equation}\label{eqn:digamma}
\psi(x)=-\gamma + (x-1) \sum_{n=0}^{\infty}\frac{1}{(1+n)(n+x)}=
-\gamma - \frac{1}{x} + \sum_{n=1}^{\infty}\frac{x}{n(n+x)}
\end{equation}
\noindent
where $\gamma$ is the Euler-Mascheroni's constant given by\\
\begin{equation}\label{eqn:Eulers-constant}
\gamma=\lim_{n \rightarrow \infty} \left( \sum_{k=1}^{n}\frac{1}{k}- \ln n \right)=-\psi(1)= 0.577215664... \, .
\end{equation}

\noindent
Let $\Gamma_q(x)$ be the $q$-extension (also known as, $q$-analogue, $q$-deformation or  $q$-generalization) of the Gamma function defined for $x>0$ and for fixed $q\in(0,1)$ by (see \cite{Askey-1978}, \cite{Kim-Rim-2000} and the references therein).

\begin{equation*}
\Gamma_q(x) = (1-q)^{1-x}\prod_{n=0}^{\infty}\frac{1-q^{n+1}}{1-q^{n+x}}=(1-q)^{1-x}\prod_{n=1}^{\infty}\frac{1-q^{n}}{1-q^{n+x}}.
\end{equation*}

\noindent
Similarly, $\Gamma_q(x)$ satisfies the following properties \cite{Chung-Kim-Mansour-2014}.
\begin{align}
\Gamma_q(x+1)&=[x]_q\Gamma_q(x),   \quad x\in R^{+}\label{eqn:functional-eqn-q-Gamma-reals}\\
\Gamma_q(n+1)&=[n]_q!,  \quad n\in Z^{+} \label{eqn:functional-eqn-q-Gamma-integers}
\end{align}
where $[x]_q=\frac{1-q^{x}}{1-q}$.  Note that equations ~(\ref{eqn:functional-eqn-q-Gamma-reals}) and ~(\ref{eqn:functional-eqn-q-Gamma-integers}) are respectively the $q$-extensions of equations ~(\ref{eqn:functional-eqn-Gamma-reals}) and ~(\ref{eqn:functional-eqn-Gamma-integers}).\\

\noindent
Likewise, the $q$-extension of the digamma function is  defined for $x>0$ and $q\in(0,1)$ as the logarithmic derivative of the function $\Gamma_q(x)$. That is,
\begin{equation*}\label{eqn:q-Psi}
\psi_q(x)=\frac{d}{dx}\ln \Gamma_q(x)=\frac{\Gamma'_q(x)}{\Gamma_q(x)}.
\end{equation*}
\noindent
It also exhibits the following series representations (see \cite{Ismail-Muldoon-2013},  \cite{Krasniqi-Mansour-Shabani-2010} and the related references therein).
\begin{align}
\psi_q(x)&= -\ln(1-q) + (\ln q) \sum_{n=0}^{\infty}\frac{q^{n+x}}{1-q^{n+x}} \nonumber \\
&=-\ln(1-q) + (\ln q)  \sum_{n=1}^{\infty}\frac{q^{nx}}{1-q^{n}} \label{eqn:q-digamma}
\end{align}

\noindent
The function $\psi_q(x)$ is increasing for $x>0$ \cite[Lemma 2.2]{Shabani-2008}. Also, for $q>0$ and $x>0$, $\psi_q(x)$ has a uniquely determined positive root \cite[Lemma 4.5]{Alzer-Grinshpan-2007}.\\

\noindent  
Further, let $\gamma_q$ be the $q$-extension of the Euler-Mascheroni's constant (see \cite{Elmonser-Brahim-Fitouhi-2012}, \cite{Ernst-2009}, \cite{Krattenthaler-Srivastava-1996}). Then,    
\begin{equation*}\label{eqn:q-Eulers-constant}
\gamma_q=\lim_{n \rightarrow \infty} \left( \sum_{k=1}^{n}\frac{1}{[k]_q}- \ln [n]_q \right)=-\psi_q(1).
\end{equation*}

\noindent
The following limit relations are valid (see  \cite{Ernst-2009},  \cite{Krattenthaler-Srivastava-1996}, \cite{Mansour-2008}). 
\begin{equation*}\label{eqn:q-Eulers-constant}
\lim_{q \rightarrow 1} \Gamma_q(x)=\Gamma(x), \quad \lim_{q \rightarrow 1} \psi_q(x)=\psi(x) \quad \text{and}\quad \lim_{q \rightarrow 1} \gamma_q=\gamma .
\end{equation*}

\begin{remark}
Unlike the value of $\gamma$ which is fixed, the value $\gamma_q$ varies according to the value of $q$. Tables of some approximate values of $\gamma_q$ can be found in \cite{Elmonser-Brahim-Fitouhi-2012} and \cite{Ernst-2009}.\\
\end{remark}

\noindent
By taking the $m$-th derivative of   ~(\ref{eqn:q-digamma}), it can easily be shown that
\begin{equation*}
\psi^{(m)}_q(x) = (\ln q)^{m+1}  \sum_{n=1}^{\infty}\frac{n^m q^{nx}}{1-q^{n}}
\end{equation*}
for $m\geq 1$. The functions $\psi^{(m)}_q(x)$ are  called the $q$-extension of the polygamma functions.

\begin{definition}\label{def:geometric-convexity}
(\cite{Krasniqi-Shabani-2010}, \cite{Nilculescu-2000}, \cite{Zhang-Xu-Situ-2007}). Let $f:I\subseteq(0,\infty) \rightarrow (0,\infty)$ be a continuous function. Then $f$ is called \textit{geometrically (or multiplicatively) convex}  on $I$ if there exists $n\ge2$ such that one of the following two inequalities holds:

\begin{equation}\label{eqn:ineq-geometric-convexity-1}
f(\sqrt{x_1x_2})\leq \sqrt{f(x_1)f(x_2)},
\end{equation}
\begin{equation}\label{eqn:ineq-geometric-convexity-2}
f\left( \prod_{i=1}^{n}x_{i}^{\lambda_i}\right) \leq \prod_{i=1}^{n}\left[ f(x_i) \right]^{\lambda_i}
\end{equation}
where $x_1, x_2, \dots, x_n \in I$ and $\lambda_1, \lambda_2, \dots, \lambda_n >0$ with $\sum_{i=1}^{n}\lambda_i = 1$. If inequalities ~(\ref{eqn:ineq-geometric-convexity-1}) and ~(\ref{eqn:ineq-geometric-convexity-2}) are reversed, then $f$ is called \textit{geometrically (or multiplicatively) concave} on $I$.
\end{definition}

\noindent
In 1971, Ke\v{c}ki\'{c}  and Vasi\'{c}  \cite[Theorem 1]{Keckic-Vasic-1971} established the double inequality  
\begin{equation}\label{eqn:ineq-Keckic-Vasic}
\frac{x^{x-1}e^{y}}{y^{y-1}e^{x}}\leq \frac{\Gamma(x)}{\Gamma(y)}\leq \frac{x^{x-\frac{1}{2}}e^{y}}{y^{y-\frac{1}{2}}e^{x}}
\end{equation}
for $x\geq y>1$,  by employing the monotonicity properties of certain functions involving the Gamma function.\\

\noindent
Also, in 2007,  Zhang, Xu and Situ \cite[Theorem 1.2]{Zhang-Xu-Situ-2007} established the double inequality
\begin{equation}\label{eqn:ineq-Zhang-Xu-Situ}
\frac{x^x}{y^y}\left( \frac{x}{y}\right)^{y[\psi(y)-\ln y]}e^{y-x}  \leq \frac{\Gamma(x)}{\Gamma(y)} \leq \frac{x^x}{y^y}\left( \frac{x}{y}\right)^{x[\psi(x)-\ln x]}e^{y-x} 
\end{equation}
for $x>0$ and $y>0$, by using the geometric convexity of a certain function related to the gamma function, and as a byproduct, inequality  ~(\ref{eqn:ineq-Keckic-Vasic}) was recovered.\\

\noindent
Furthermore, in 2010, Krasniqi and Shabani \cite[Theorem 3.5]{Krasniqi-Shabani-2010} also established the following related inequality for the $p$-Gamma function.
\begin{equation}\label{eqn:ineq-Krasniqi-Shabani}
\left( \frac{x}{y}\right)^{y[1+\psi_p(y)]}e^{y-x}  \leq \frac{\Gamma_p(x)}{\Gamma_p(y)} \leq \left( \frac{x}{y}\right)^{x[1+ \psi_p(x)]}e^{y-x}
\end{equation}
for $x>0$ and $y>0$.\\

\noindent
For more information  on inequalities of this nature,  one could access  the review article by Qi \cite{Qi-2010-Hindawi}.\\

\begin{lemma}\label{lem:geometric-convexity-lemma-1}
Let $f:I\subseteq(0,\infty) \rightarrow (0,\infty)$ be a differentiable function. Then $f$ is a geometrically convex function if and only if the function $\frac{xf'(x)}{f(x)}$ is nondecreasing.
\end{lemma}

\begin{lemma}\label{lem:geometric-convexity-lemma-2}
Let $f:I\subseteq(0,\infty) \rightarrow (0,\infty)$ be a differentiable function. Then $f$ is a geometrically convex function if and only if the function $\frac{f(x)}{f(y)} \geq \left( \frac{x}{y}\right)^{\frac{yf'(y)}{f(y)}}$ holds for any $x,y\in I$.
\end{lemma}

\noindent
For proofs of Lemmas \ref{lem:geometric-convexity-lemma-1} and \ref{lem:geometric-convexity-lemma-2}, see \cite{Nilculescu-2000}.\\

\noindent
The purpose of this paper is to establish some related inequalities for the $q$-extension of the Gamma function, by using  geometric convexity and monotonicity features of certain functions associated with  the $q$-extended Gamma function. We present our results in the following section.  \\


\section{Main Results}

\begin{theorem}\label{thm:q-of-Krasniqi-Shabani}
Let $x\ge1$, $y\ge1$ and $q\in(0,1)$. Then the  double inequality
\begin{equation}\label{eqn:q-of-Krasniqi-Shabani} 
\left( \frac{x}{y}\right)^{y(-\frac{\ln q}{1-q}q^{y}+\psi_q(y))}e^{\left( \frac{q^x - q^y}{1-q}\right)}  \leq \frac{\Gamma_q(x)}{\Gamma_q(y)}  \leq  \left( \frac{x}{y}\right)^{x(-\frac{\ln q}{1-q}q^{x}+\psi_q(x))}e^{\left( \frac{q^x - q^y}{1-q}\right)}
\end{equation}
holds true.
\end{theorem}

\begin{proof}
Define a function $f$ for $x\ge1$ and $q\in(0,1)$ by $f(x)=e^{[x]_q}\Gamma_q(x)$. Then,
\begin{equation*}
\ln f(x)=[x]_q + \ln\Gamma_q(x) \quad \text{which implies}\quad \frac{f'(x)}{f(x)}=[x]'_q + \psi_q(x) = -\frac{(\ln q)q^{x}}{1-q} + \psi_q(x).
\end{equation*}
That further implies,
\begin{align*}
\left( \frac{xf'(x)}{f(x)}\right)&= -x\frac{(\ln q)q^{x}}{1-q} + x\psi_q(x) \qquad \text{yielding,}\\
\left( \frac{xf'(x)}{f(x)}\right)'&= -\left( \frac{(\ln q)q^{x}}{1-q} + x \frac{(\ln q)^{2}q^{x}}{1-q} \right) +  \psi_q(x) + x\psi'_q(x) \\
&=-\frac{(\ln q)q^{x}}{1-q} - x \frac{(\ln q)^{2}q^{x}}{1-q} - \ln(1-q) + (\ln q) \sum_{n=1}^{\infty}\frac{q^{nx}}{1-q^{n}} \\
& \quad  + x (\ln q)^{2} \sum_{n=1}^{\infty}\frac{nq^{nx}}{1-q^{n}} \\
&= - \ln(1-q) + (\ln q) \sum_{n=2}^{\infty}\frac{q^{nx}}{1-q^{n}} +  x(\ln q)^{2} \sum_{n=2}^{\infty}\frac{nq^{nx}}{1-q^{n}} \\
&= - \ln(1-q) + \sum_{n=2}^{\infty}\left[ \frac{(\ln q)q^{nx} + nx(\ln q)^{2} q^{nx}}{1-q^{n}} \right]\ge0.
\end{align*}
\noindent
Thus $\frac{xf'(x)}{f(x)}$ is nondecreasing. Therefore, by Lemmas \ref{lem:geometric-convexity-lemma-1} and \ref{lem:geometric-convexity-lemma-2}, $f$ is geometrically convex, and consequently \, $\frac{f(x)}{f(y)} \geq \left( \frac{x}{y}\right)^{\frac{yf'(y)}{f(y)}}$ \, resulting to

\begin{equation}\label{eqn:left-ineq}
 \frac{e^{[x]_q}\Gamma_q(x)}{e^{[y]_q}\Gamma_q(y)} \geq \left( \frac{x}{y}\right)^{y([y]'_q + \psi_q(y))}
\end{equation}
and
\begin{equation} \label{eqn:right-ineq}
 \frac{e^{[y]_q}\Gamma_q(y)}{e^{[x]_q}\Gamma_q(x)} \geq \left( \frac{y}{x}\right)^{x([x]'_q + \psi_q(x))}. 
\end{equation}
Combining  ~(\ref{eqn:left-ineq}) and ~(\ref{eqn:right-ineq}) concludes  the proof of Theorem \ref{thm:q-of-Krasniqi-Shabani}. Observe that $[y]_q - [x]_q =\frac{q^x - q^y}{1-q}$.\\
\end{proof}

\begin{corollary}\label{cor:Cor-of-Krasniqi-Shabani}
For $x>0$ and $q\in(0,1)$, the  inequalities
\begin{multline}\label{eqn:Cor-of-Krasniqi-Shabani} 
\left( \frac{x+1}{x+\frac{1}{2}}\right)^{(x+\frac{1}{2})\left(-\frac{\ln q}{1-q}q^{(x+\frac{1}{2})}+\psi_q(x+\frac{1}{2})\right)}e^{q^{x}\left( \frac{q - \sqrt{q}}{1-q}\right)}  \leq \frac{\Gamma_q(x+1)}{\Gamma_q(x+\frac{1}{2})} \\   \leq    \left( \frac{x+1}{x+\frac{1}{2}}\right)^{(x+1)\left(-\frac{\ln q}{1-q}q^{(x+1)}+\psi_q(x+1)\right)}e^{q^{x}\left( \frac{q - \sqrt{q}}{1-q}\right)} 
\end{multline}
are valid.
\end{corollary}

\begin{proof}
This follows directly from Theorem \ref{thm:q-of-Krasniqi-Shabani} by substituting $x$ by $x+1$, and $y$ by $x+\frac{1}{2}$.
\end{proof}


\begin{theorem}\label{thm:q-extension-ineq}
Let $x>0$, $y>0$ and $\alpha \ge x^*$, where $x^*$ is the unique positive root of $\psi_q(x)$. Then for fixed $q\in(0,1)$,  the  double inequality
\begin{equation}\label{eqn:q-extension-ineq} 
e^{y-x}\frac{(x+\alpha)}{(y+\alpha)}\left( \frac{x}{y}\right)^{y\left(\frac{y+\alpha-1}{y+\alpha} + \psi_q(y+\alpha)\right)}  \leq \frac{\Gamma_q(x+\alpha)}{\Gamma_q(y+\alpha)}   \leq 
e^{y-x}\frac{(x+\alpha)}{(y+\alpha)}\left( \frac{x}{y}\right)^{x\left(\frac{x+\alpha-1}{x+\alpha} + \psi_q(x+\alpha)\right)}
\end{equation}
holds true.
\end{theorem}

\begin{proof}
Define a function $g$ for $x>0$ and $q\in(0,1)$ by $g(x)=\frac{e^x\Gamma_q(x+\alpha)}{x+\alpha}$. Then,
\begin{equation*}
\ln g(x)=x + \ln\Gamma_q(x+\alpha)-\ln(x+\alpha). \quad \text{This implies}\quad \frac{g'(x)}{g(x)}=1 + \psi_q(x+\alpha) -\frac{1}{x+\alpha}.
\end{equation*}
Then,
\begin{equation*}
\frac{xg'(x)}{g(x)}= x + x\psi_q(x+\alpha) -\frac{x}{x+\alpha},
\end{equation*}
from which we obtain,
\begin{equation*}
\left( \frac{xg'(x)}{g(x)}\right)'= 1 + \psi_q(x+\alpha) + x \psi'_q(x+\alpha) - \frac{1}{(x+\alpha)^2} > 0.
\end{equation*}
\noindent
Thus $\frac{xg'(x)}{g(x)}$ is nondecreasing. Therefore, by Lemmas \ref{lem:geometric-convexity-lemma-1} and \ref{lem:geometric-convexity-lemma-2}, $g$ is geometrically convex, and thus \, $\frac{g(x)}{g(y)} \geq \left( \frac{x}{y}\right)^{\frac{yg'(y)}{g(y)}}$. Consequently, we obtain
\begin{equation*} \label{eqn:geo-ineq-1}
\frac{(y+\alpha)e^x\Gamma_q(x+\alpha)}{(x+\alpha)e^y\Gamma_q(y+\alpha)}   \geq  \left( \frac{x}{y}\right)^{y \left(\frac{y+\alpha-1}{y+\alpha} + \psi_q(y+\alpha) \right)}
\end{equation*}
and
\begin{equation*} \label{eqn:geo-ineq-2}
\frac{(x+\alpha)e^y\Gamma_q(y+\alpha)}{(y+\alpha)e^x\Gamma_q(x+\alpha)}   \geq  \left( \frac{y}{x}\right)^{x \left(\frac{x+\alpha-1}{x+\alpha} + \psi_q(x+\alpha) \right)}
\end{equation*}
concluding  the proof of Theorem \ref{thm:q-extension-ineq}.\\
\end{proof}

\begin{theorem}\label{thm:q-extension-ineq-using-MVT}
Let $x>y>0$. Then for fixed $q\in(0,1)$,  the  double inequality
\begin{equation}\label{eqn:q-extension-ineq-using-MVT} 
e^{(x-y)\psi_q(y)} < \frac{\Gamma_q(x)}{\Gamma_q(y)} < e^{(x-y)\psi_q(x)} 
\end{equation}
holds true.
\end{theorem}

\begin{proof}
Define a function $h_q$ for $t>0$ and $q\in(0,1)$ by $h_q(t)=\ln \Gamma_q(t)$. Let $(y,x)$ be fixed. Then, by the well-known mean value theorem, there exists a $c\in(y,x)$ such that
\begin{equation*}
h_q'(c)= \frac{\ln \Gamma_q(x) - \ln \Gamma_q(y)}{x-y},
\end{equation*}
implying,
\begin{equation*}
\psi_q(c)= \frac{1}{x-y}\ln \frac{ \Gamma_q(x)}{ \Gamma_q(y)}.
\end{equation*}
\noindent
Since $\psi_q(t)$ is increasing for $t>0$, then for $c\in(y,x)$ we obtain
\begin{equation*} 
\psi_q(y) < \frac{1}{x-y}\ln \frac{ \Gamma_q(x)}{ \Gamma_q(y)} < \psi_q(x). 
\end{equation*}
That is
\begin{equation*}
(x-y) \psi_q(y) < \ln \frac{ \Gamma_q(x)}{ \Gamma_q(y)} < (x-y) \psi_q(x) .
\end{equation*}
Exponentiating yields the desired results.\\
\end{proof}

\begin{remark} 
The double inequality ~(\ref{eqn:q-extension-ineq-using-MVT}) provides the $q$-extension of \cite[Corollary 1.5]{Zhang-Xu-Situ-2007} and  \cite[Corollary 2]{Qi-2002-MIA}. 
\end{remark}

\begin{corollary}\label{cor:q-extension-ineq-using-MVT-1}
For $x>0$, $\mu > \lambda >0$ and $q\in(0,1)$, the  inequalities
\begin{equation}\label{eqn:q-extension-ineq-using-MVT-1} 
e^{(\mu - \lambda)\psi_q(x + \lambda)}  <  \frac{\Gamma_q(x + \mu)}{\Gamma_q(x + \lambda)}   <  e^{(\mu - \lambda)\psi_q(x + \mu)}  
\end{equation}
hold true.
\end{corollary}

\begin{proof}
This follows directly from Theorem \ref{thm:q-extension-ineq-using-MVT} by substituting $x$ by $x+\mu$, and $y$ by $x+\lambda$.
\end{proof}

\begin{remark} 
If we set $\mu=1$ in Corollary \ref{cor:q-extension-ineq-using-MVT-1}, then we obtain the $q$-extension of the result of Laforgia and Natalini \cite[Theorem 3.1]{Laforgia-Natalini-2013-ADSA}.
\end{remark}

\begin{corollary}\label{cor:q-extension-ineq-using-MVT-2}
For $x>0$ and $q\in(0,1)$, the  inequalities
\begin{equation}\label{eqn:q-extension-ineq-using-MVT-2} 
e^{\frac{1}{2} \psi_q(x + \frac{1}{2})}  <  \frac{\Gamma_q(x + 1)}{\Gamma_q(x + \frac{1}{2})}   <  e^{\frac{1}{2} \psi_q(x + 1)}  
\end{equation}
hols true.
\end{corollary}

\begin{proof}
Follows from Theorem \ref{thm:q-extension-ineq-using-MVT} by substituting $x$ by $x+1$, and $y$ by $x+\frac{1}{2}$.
\end{proof}

\begin{remark} 
By virtue of relation  ~(\ref{eqn:functional-eqn-q-Gamma-reals}), inequalities ~(\ref{eqn:q-extension-ineq-using-MVT-2}) can be rearranged as
\begin{equation}
\left( \frac{1-q}{1-q^x}\right) e^{\frac{1}{2} \psi_q(x + \frac{1}{2})}  <  \frac{\Gamma_q(x)}{\Gamma_q(x + \frac{1}{2})}   <  \left( \frac{1-q}{1-q^x}\right) e^{\frac{1}{2} \psi_q(x + 1)}.  
\end{equation}
\end{remark}

\begin{remark} 
Results similar to inequalities  ~(\ref{eqn:Cor-of-Krasniqi-Shabani}) and  ~(\ref{eqn:q-extension-ineq-using-MVT-2})  can also be found in  \cite{Gao-2011},  \cite{Gao-2014-MIA} and \cite{Nant-Prem-2014-Issues}.
\end{remark}

\section{Conclusion}

In the paper, the authors have established some inequalities for the $q$-extension of the classical Gamma function. The results provide generalizations for several previous results. The findings of this research could provide useful information for researchers interested in $q$-analysis in particular, and the theory of inequalities in general. In addition, a further research could  be conducted to see if similar results could be obtained for other special functions like the $q$-Beta and $q$-Psi functions. This could further expand the potential applications of our results.

\section*{Competing Interests}
The authors declare that there is no competing interests with regard to the publication of this manuscript.


\section*{Acknowledgements}
The authors are grateful to the anonymous reviewers  for their useful comments and suggestions, which helped in improving the quality of this paper.

\bibliographystyle{plain}


\end{document}